\documentclass[draft]{amsart}
\usepackage{amssymb,amscd}
\usepackage[all]{xy}

\newcommand{\Z}{\mathbb Z}
\newcommand{\Q}{\mathbb Q}

\newcommand{\C}{\mathbb C}
\newcommand{\N}{\mathbb N}

\newcommand{\cA}{\mathcal A}
\newcommand{\cB}{\mathcal B}

\newcommand{\cD}{\mathcal D}
\newcommand{\cE}{\mathcal E}
\newcommand{\cF}{\mathcal F}

\newcommand{\cL}{\mathcal L}
\newcommand{\cM}{\mathcal M}
\newcommand{\cN}{\mathcal N}
\newcommand{\cO}{\mathcal O}

\newcommand{\cV}{\mathcal V}

\renewcommand{\t}{\widetilde}

\newcommand{\tX}{\widetilde X}

\renewcommand{\frak}{\mathfrak}
\newcommand{\m}{\mathfrak m}

\DeclareMathOperator{\Hom}{Hom}

\DeclareMathOperator{\specan}{Specan}

\DeclareMathOperator{\di}{div}

\DeclareMathOperator{\LF}{LF}

\DeclareMathOperator{\mult}{mult}

\renewcommand{\:}{\colon}

\newcommand{\vord}[1]{#1\text{-}\mathrm{ord}}

\newcommand{\gen}[1]{\langle #1 \rangle}

\newcommand{\fl}[1]{#1^{\flat}}

\newcommand{\defset}[2]{{\left\{#1\,\left| \,#2 \right. \right\}}}
\newcommand{\X}{(X,o)}
\newcommand{\abX}{(X^{u},o)}
\newcommand{\Y}{(Y,o)}

\newcommand{\ten}{\circle*{0.25}}

 \theoremstyle{plain}
\newtheorem{thm}{Theorem}[section]
\newtheorem{cor}[thm]{Corollary}
\newtheorem{lem}[thm]{Lemma}
\newtheorem{prop}[thm]{Proposition}

 \theoremstyle{definition}
\newtheorem{defn}[thm]{Definition}

\newtheorem{ex}[thm]{Example}

\theoremstyle{remark}

\newtheorem{rem}[thm]{Remark}

\numberwithin{equation}{section}

\newcommand{\thmref}[1]{Theorem~\ref{#1}}
\newcommand{\lemref}[1]{Lemma~\ref{#1}}
\newcommand{\corref}[1]{Corollary~\ref{#1}}
\newcommand{\proref}[1]{Proposition~\ref{#1}}
\newcommand{\remref}[1]{Remark~\ref{#1}}

\newcommand{\defref}[1]{Definition~\ref{#1}}

\newcommand{\sref}[1]{Section~\ref{#1}}

\begin{document}

\title[The multiplicity of abelian covers]{The multiplicity of abelian covers of splice quotient singularities}

\author{Tomohiro Okuma}

\address{Department of Education, Yamagata University, 
 Yamagata 990-8560, Japan.}
  \curraddr{Department of Mathematical Sciences, 
Yamagata University, 
 Yamagata 990-8560, Japan.}
\email{okuma@sci.kj.yamagata-u.ac.jp}

\dedicatory{Dedicated to Professor Tadashi Tomaru
on the occasion of his sixtieth birthday}
 \date{May 2, 2011}
\thanks{This work was partly supported by KAKENHI 20540060.
}



\begin{abstract}
From a resolution graph with certain conditions,  Neumann and
 Wahl constructed 
an equisingular family of surface singularities called
 splice quotients. For this class some fundamental analytic
 invariants have been computed from their resolution graph.
In this paper we give a method to compute the multiplicity of an abelian covering  of
 a splice quotient  from its resolution graph
 and the Galois group.
\end{abstract}

\maketitle

\section{Introduction}

Splice quotient  singularities were introduced by Neumann
and Wahl (\cite{nw-uac}, \cite{nw-CIuac}, \cite{nw-HSL}), which are a broad generalization of
quasihomogeneous complex surface singularities with rational 
homology sphere links. 

Let $\X$ be a normal complex surface singularity with  rational 
homology sphere link and $\Gamma$ the weighted dual graph of
$\X$, i.e., the  weighted dual graph
of the exceptional set of the minimal good resolution. 
It is well-known that $\Gamma$  determines the topology of
$\X$, and the converse is also true. 
Our main issue is to clarify which analytic invariants can be computed from $\Gamma$, and to give a method to compute them.
Let $\varphi\: (X^u,o) \to \X$ denote the {\itshape universal abelian cover}.
Under a certain condition on $\Gamma$, 
 Neumann and Wahl gave a method to construct explicitly the so-called
 ``splice diagram equations''
 from $\Gamma$.
These equations define a complete intersection normal surface 
 singularity $Z$  homeomorphic to $X^u$, 
 having an action of the Galois group of
 $\varphi$ with quotient $Y$  homeomorphic to $X$; 
the singularity $Y$ is called a {\itshape splice quotient}.
It is very natural to expect that some  analytic invariants of splice
quotients can be computed from their weighted dual graphs.
In fact, since the end curve theorem (a characterization of
splice quotients) was given by Neumann--Wahl \cite{nw-ECT}, 
we have been able to compute
the geometric genus (\cite{o.pg-splice}), the dimension of cohomology
groups of certain invertible sheaves (\cite{o.pg-splice}, \cite{no-CIC}, \cite{no-SW}, \cite{nem.coh-sq}), and the multiplicity (\cite{nem.coh-sq}) of splice quotients from their weighted dual graphs.
However, we should remark that the embedding dimension, which is one of the most fundamental analytic invariants, is not topological even for quasihomogeneous singularities.

In this article, we generalize the results on the multiplicity. Namely, 
we consider an arbitrary abelian covering
 $(X',o) \to \X$ of a splice quotient $\X$ such that
 $X'\setminus\{o\} \to X\setminus \{o\}$ is unramified, and
  we give an explicit method  to compute the multiplicity of
 $(X',o)$  from $\Gamma$  and the Galois group $G'$ of $X'\to X$. 
Consequently, the multiplicity of $(X,o)$ and $(X^u,o)$ can be computed from $\Gamma$.
Note that we do not need to compute the weighted dual graph $\Gamma'$
 of $(X',o)$;  it seems difficult to obtain $\Gamma'$ from $\Gamma$ and $G'$.
The outline of our approach is as follows.
If the ideal sheaf $I_Z$ of the maximal ideal cycle $Z$ on a good resolution $\tX$ of $X$ is generated by global sections, then the multiplicity of $\X$ is $-Z\cdot Z$, where $Z\cdot Z$ denotes the self-intersection number of $Z$. 
Let $\t X'$ be a good resolution of $(X',o)$ with a morphism $\psi\: \tX' \to \tX$ that naturally induces the covering $X'\setminus\{o\} \to X\setminus \{o\}$. 
The end curve theorem enables us to determine whether required functions on $\tX$ (resp. $\t X'$) exist by $\Gamma$ (resp. $\Gamma$ and $G'$).
Then we can find a rational cycle $Z'$ on $\tX$ such that 
$\psi^*Z'$ is the maximal ideal cycle on $\tX'$, and blowing-ups that resolve the base points of the ideal sheaf. 
Thus, on an appropriate good resolution, 
 the multiplicity of $(X',o)$ can be computed as $|G'|(-Z'\cdot Z')$, where $|G'|$ denotes the order of the group $G'$.  

The multiplicity of these singularities can also be obtained by
the theory of Gr{\"o}bner basis (one needs to compute the equations for $(X',o)$ from splice diagram equations with group action).  
However our method is rather based on linear algebra 
and combinatorics on graphs and it may be easier to apply.
We will use
 ``monomial cycles'' to connect directly the combinatorics of
the  graph and the complex structure of the singularity.

This article is organized as follows.
\sref{s:abeliancover} summarizes basic facts on abelian coverings of
a surface singularity. In \sref{s:max-mult}, we discuss the maximal ideal cycle 
on a resolution  of an abelian covering in terms
of rational cycles on  a resolution  of the underlying singularity.
In \sref{s:SQ}, we give a brief introduction of splice
quotients and a lemma on generators of certain ideals 
 associated with a cycles on a resolution of a splice quotient.
The main theorem is proved in \sref{s:mult}; the point here
is to control the base points of the
ideal sheaf on a resolution generated by the maximal
ideal of the singular point. 
In the last section, applying our method,
 we compute some concrete examples.

The author thanks the referees for their helpful comments and suggestions that greatly improved this paper.

\section{Abelian covers}\label{s:abeliancover}

Let $\X$ be a normal complex surface singularity with
rational homology sphere link $\Sigma$.
We may assume that $X$ is homeomorphic to the cone over the
link $\Sigma$.
Let $\pi\:\tX \to X$ be a good resolution and  
 $\{E_v\}_{v \in \cV}$ the set of irreducible components
 of  the exceptional
divisor $E:=\pi^{-1}(o)$.
By the assumption,  $E$ is a tree of rational curves.
We call an element of the group
$$
L:=\sum _{v\in \cV}\Z E_v
$$
a cycle. 
Since the intersection matrix
$I(E):=(E_v\cdot E_w)$ is negative definite, for every $v\in
\cV$ there exists an effective $\Q$-cycle $E_v^*\in L\otimes \Q$ such that
$E_v^*\cdot E_w=-\delta_{vw}$,
where $\delta_{vw}$ denotes the Kronecker delta.
We define abelian groups as follows:
 $$
L^*:=\sum_{v \in \cV}\Z E_v^* \subset L\otimes \Q, \quad H:=L^*/L.
 $$
Then $H\cong H_1(\Sigma,\Z)$.

\begin{defn}
We call a finite morphism $\Y \to \X$ of  normal surface
 singularities   
an {\itshape abelian covering} if it induces  an
 unramified abelian covering $Y\setminus\{o\}\to
 X\setminus\{o\}$.
We denote the Galois group of the covering by $G(Y/X)$.
The {\itshape universal abelian
 covering} is the abelian covering  $\abX \to \X$ with
 $G(X^{u}/X)=H$.
\end{defn}
There is a one-to-one correspondence between abelian covers
of $\X$ and subgroups of $H$.
Next we give a description of the abelian covering in
terms of eigensheaves (see \cite{o.pg-splice} or 
\cite{o.uac-rat} for details).

Let $\hat H=\Hom (H, \C^*)$. Then an isomorphism $\theta\: H \to
\hat H$ is defined by 
$$
\theta(h)(h')=\exp(2\pi\sqrt{-1}h\cdot h'), \quad h, h'\in H,
$$
where the  pairing $(h,h')\mapsto h\cdot h' \in \Q/\Z$ is determined by
the matrix $I(E)$.
Let $\Theta$ denote the composite of the natural map 
$L^*\to H$ and  $\theta$.
For any divisor $D$ on $\tX$, we define
$$
c_1(D)=\sum_{v\in \cV}(-D\cdot E_v)E_v^* \in L^*.
$$
 There exists a set  $\{L_{\chi}\}_{\chi \in \hat H}$ of
effective divisors on $\tX$ with  properties:
 \begin{enumerate}
  \item  $\Theta(c_1(L_{\chi}))=\chi$;
  \item  $|H|L_{\chi}\sim c_1(|H|L_{\chi})$;
  \item $[c_1(L_{\chi})]=0$, where $[ D ]$ denotes the
	integral part of a $\Q$-divisor $D$.
 \end{enumerate} 
Since $\mathrm{Pic}(\tX)$ has no torsion elements,
 such a divisor $L_{\chi}$ is uniquely determined up to
linear equivalence. For every $D\in L^*$ we set
$$
\sigma(D)=L_{\Theta(D)}+[D].
$$
Note that $|H|D\sim\sigma(|H|D)$ for any $D \in L^*$. 
Let $\cL_{\chi}=\cO_{\tX}(-L_{\chi})$.

\begin{prop}[{cf. \cite[\S 3.2]{o.uac-rat}}]\label{p:uac-rat}
We have a  collection $\{\cL_{\chi} \otimes \cL_{\chi'} \to
 \cL_{\chi\chi'}\}$  of homomorphisms
which defines the $\cO_{\tX}$-algebra structure of
an $\cO_{\tX}$-module $\cA:=\bigoplus _{\chi \in 
 \hat H}\cL_{\chi}$ satisfying  the following:
\begin{enumerate}
 \item The projection $\specan _{X}\pi_*\cA \to X$
       coincides with $q\:X^{u} \to X$.
 \item $\tX^{u}:=\specan _{\tX}\cA$  has only cyclic quotient
       singularity and there is a morphism $\rho\: \tX^u \to
       {X^u}$ which is a partial resolution, and the following
       diagram is commutative:
$$
\begin{CD}
 \t X^u @>{p}>> \tX \\
@V{\rho}VV     @VV{\pi}V \\
X^u @>>{q}> X
\end{CD}
$$
 \item The module $\cL_{\chi}$
       (resp. $\pi_*\cL_{\chi}$) is the
       $\chi$-eigensheaf of $p_*\cO_{\tX^u}$ (resp. $q_*\cO_{{X^u}}$).
\end{enumerate}
\end{prop}

Let $q_1\:(X_1,o) \to \X$ be an abelian covering with Galois
group $G_1$.
Then there exists a subgroup $H_1$ of $H$ such that
$X_1={X^u}/H_1$ and $G_1=H/H_1$.
Let $\tX_1=\tX^u/H_1$. 
We have the following commutative diagram:
$$
\xymatrix{
  \tX^u \ar[r] \ar[d]^{\rho}& \tX_1  \ar[r]^{p_1} \ar[d]^{\pi_1} & 
 \tX   \ar[d]^{\pi} \\
 {X^u} \ar[r] & X_1  \ar[r]_{q_1} & 
 X  
}
$$
 For any subgroup $V\subset H$, we define groups $V^{\bot}$
 and $V^{\flat}$ as follows:
$$
V^{\bot}=\defset{\chi \in \hat H}{\chi(V)=\{1\}}, \quad 
  V^{\flat}=\defset{h \in H}{V\cdot h=\{0\}}.
$$
We easily see the following:
\begin{enumerate}
 \item $ H/H_1^{\flat}\cong \Hom(H_1, \Q/\Z)\cong \hat
 H_1:=\Hom(H_1, \C^*)\cong H_1$.
 \item  $\fl{(\fl{H_1})}=H_1$.
\end{enumerate}

\begin{lem}
We have 
${p_1}_*\cO_{\tX_1}=\bigoplus _{\chi \in
 H_1^{\bot}}\cL_{\chi}$  and $G(X_1/X)\cong H_1^{\flat}$.
\end{lem}
\begin{proof}
For any $h\in H$ and $f\in H^0(\cL_{\chi})$, we have
 $h\cdot f=\chi(h)f$.
Hence  $f$ is invariant under $H_1$-action if and only
 if $\chi(h)=1$ for all $h\in H_1$.
\end{proof}

\begin{rem}
 The geometric genus $p_g(X_1,o)$ of $(X_1,o)$ is expressed
 as 
$$
h^1(\cO_{\tX_1})=\sum _{\chi \in
 H_1^{\bot}}h^1(\cL_{\chi}).
$$
In case $X$ is a splice quotient, the invariant
 $h^1(\cL_{\chi})$ can be computed from $\Gamma$ (see \cite[4.5]{o.pg-splice}, or \cite[\S 2]{nem.coh-sq}), and thus
 $p_g(X_1,o)$ can be computed from $\Gamma$ and $H_1$.
\end{rem}

Let $E^1$ denote the exceptional set of $\tX_1$, and  $L^1$
the group of cycles supported in $E^1$.

\begin{defn}
If $D\in L^*$, $f\in
 H^0(\cO_{\tX}(-\sigma(D)))\setminus\{0\}$,
and $C:=\di(f)-\sigma (D)$ has no component of
 $E$, then we write $(f)_E=D$.
By the same manner, we define $(f)_{E^1}$ for $f\in
 H^0(\cO_{\tX_1}(-D))\setminus\{0\}$ with $D\in L^1$.
\end{defn}

For example, if $f\in H^0(\cO_{\tX})$ and $\di (f)=D+C$, where $D\in L$
and $C$ has no component of
 $E$, then $(f)_E=D$.
For $f\in H^0(\cL_{\chi})$ and  $g\in H^0(\cL_{\chi'})$
we simply denote by $fg\in H^0(\cL_{\chi\chi'})$ the product
in the algebra $H^0(\cA)$.
Then 
we have 
$(fg)_E=(f)_E+(g)_E$ (cf. \cite[2.1]{o.pg-splice}).

\begin{lem}\label{l:p_1^*zero}
Let  $\chi \in H_1^{\bot}$ and $f\in H^0(\cL_{\chi})$. 
Regarding $f$ as a regular
 function on $\tX_1$, we write it as $f_1$.
Then $(f_1)_{E^1}=p_1^*(f)_E$. 
In particular, $p_1^*D\in L^1$ for any $D\in
 \Theta^{-1}(\chi)$.
\end{lem}
\begin{proof}
Let $f^e$ denote the $e$-th power of $f$ in
 the algebra $\bigoplus _{\chi \in H_1^{\bot}}H^0(\cL_{\chi})$,
 where  $e=|H_1^{\bot}|$.
Then $f^e\in H^0(\cO_{\tX})$, $e(f)_E\in L$,
 and $f_1^e=p_1^*(f^e)$. Hence 
 $e(f_1)_{E^1}=ep_1^*(f)_E$.
\end{proof}

\section{The maximal ideal cycle and the multiplicity}\label{s:max-mult}

For any set $\cD$ of effective $\Q$-divisors on a normal variety
$W$, we denote by $\min \cD$ the set
of minimal elements of $\cD$ with respect to the order $\ge$
and by $\gcd \cD$ the maximal $\Q$-divisor on $W$ (if it exists) such that
$\gcd \cD\le D$ for every $D\in \cD$.

\begin{defn}
Let $\m_X$ denote the  maximal ideal of $\cO_{X,o}$.
 The {\itshape maximal ideal cycle} on $\tX$ is defined to be the
 effective cycle 
$$
\gcd\defset{(\pi^*f)_E}{f\in \m_X\setminus \{0\}}.
$$
Note that the maximal ideal cycle can be defined on any
 partial resolution.
\end{defn}

We shall compute the multiplicity of the singularity applying the
following (cf. \cite[2.7]{wag.ell}, \cite[4.6]{chap}).

\begin{thm}\label{t:mult}
 Let $Z$ be the maximal ideal cycle on $\tX$.
If $\cO_{\tX}(-Z)$ is generated by its global sections at
 every point of $E$, then the multiplicity $\mult\X$ of $\X$ is $-Z\cdot Z$.
\end{thm}

Assume that  $\cD$ is a subset of $L\otimes \Q$ and
  $\gcd\cD$ exists.

To express the maximal ideal cycle on the resolution space of
an abelian cover of $X$ in terms of the $\gcd$ of elements of
$L\otimes \Q$, we introduce a condition on $\cD$.
 For any open subset $U \subset \tX$, let 
 $\cD_U=\defset{D|_U}{D \in \cD}$.

\begin{defn}\label{d:GCD}
We say that $\cD$ satisfies the {\itshape GCD condition} 
at a point $x\in E$ if 
 there exists a neighborhood $U \subset \tX$ of $x$ such
 that $\min \cD_U=\{\gcd\cD_U\}$, i.e., $\gcd \cD_U\in
 \cD_U$.
We simply say that  $\cD$ satisfies the GCD condition
if  $\cD$ satisfies the GCD condition at every point $x \in E$.
\end{defn}

Clearly the condition that $\gcd \cD_U\in \cD_U$ is always
satisfied on a neighborhood of any nonsingular point of
$E$.

 For any surjective morphism $\varphi\:W\to\tX$, let
 $\varphi^*\cD=\defset{\varphi^*D}{D \in \cD}$.

\begin{ex}
Suppose that $E_1\cdot E_2=1$.
Let $D_1=2E_1+3E_2$,  $D_2=5E_1+E_2$, and 
$\cD=\{D_1, D_2\}$. 
Then $\gcd \cD=2E_1+E_2$ and $\cD$ does not satisfy the GCD condition.
Let $\varphi \: W
 \to \tX$ be the blowing-up at $E_1\cap E_2$ and $C\subset W$
 the exceptional curve.
Then
$$
 \gcd\varphi^*\cD=\varphi^*\gcd\cD+2C.
$$
\end{ex}

\begin{lem}\label{l:*gcd}
Let $\varphi\:W\to\tX$ be any surjective morphism such that
 $W$ is also a partial resolution of a surface singularity
 with exceptional set $\varphi ^{-1}(E)$.
 Assume that $\cD$ satisfies the GCD condition.
Then $\varphi^*\cD$ also satisfies the GCD condition and
 $\varphi^*\gcd\cD=\gcd\varphi^*\cD$.
\end{lem}
\begin{proof}
Let $U$ be as in \defref{d:GCD} and $\varphi_U=\varphi
 |_{\varphi ^{-1}(U)}$. Then we see that
$\min \varphi_U^*\cD_U=\varphi_U^*\min \cD_U$ and 
 $(\gcd\varphi^*\cD)|_{\varphi^{-1}(U)}=\varphi_U^*\gcd\cD|_U$. 
\end{proof}

\begin{lem}\label{l:to-gcd}
Assume that $\min \cD$ is a finite set. 
Then there exists a birational morphism $\varphi\:W\to\tX$
such that $\varphi^*\cD$ satisfies the GCD condition.
In fact,  $\varphi$ is obtained by a finite succession of blowing-ups at
       singular points of exceptional sets where the
       GCD condition is not satisfied. 
\end{lem}
\begin{proof}
In \cite[6.4]{o.uac-rat}, $\cD$ is said to be locally
 ordered if $\{D_1, D_2\}$  satisfies the
 GCD condition for  arbitrary $D_1, D_2\in \cD$.
By \cite[6.5]{o.uac-rat}, if $\cD$
 is finite then there exists a morphism $\varphi\:W\to\tX$ with  $\varphi^*\cD$ being locally ordered.
In fact, $\varphi$ is obtained by a finite succession of
 blowing-ups at singular points of exceptional sets where
 the property being  locally ordered is not satisfied.
Applying this argument to the set $\min \cD$, we obtain the assertion.
\end{proof}

We identify the local ring $\cO_{X_1,o}$ with 
$\pi_*\left(\bigoplus _{\chi \in
H_1^{\bot}}\cL_{\chi}\right)_o
$,
and let $\frak m_1$ denote the maximal ideal of  $\cO_{X_1,o}$.
Suppose that $x_1, \dots , x_m \in \frak m_1$ generate
$\frak m_1$ and $x_i \in
\left(\pi_*\cL_{\chi_i}\right)_o$ for $i=1, \dots, m$.
Let $\cD$ denote the set of  $D_i:=(\pi^*x_i)_{E}$, $i=1,
\dots, m$. Note that $\Theta (D_i)=\chi_i$.
Let $\varphi$ denote the composite
 of any resolution  $\bar X_1\to \tX_1$ of the singularities
 of $\tX_1$ and the morphism $p_1\: \tX_1\to \tX$, and $Z^1$ the
 maximal ideal cycle on $\bar X_1$. 
By \lemref{l:p_1^*zero}, we have $Z^1=\gcd\{\varphi^*D_1,
\dots, \varphi^* D_m\}$. 
Let $Z=\gcd\cD$.

\begin{prop}\label{p:mult-gen}
 Assume that for every point $x\in E$ there exist a
 neighborhood $U\subset \tX$ of $x$ and   $D_i\in \cD$
 such that $Z|_U=D_i|_U$ and  
$\cO_{\tX}(-\sigma(D_i))$ is generated by its global sections at
 every point of $U\cap E$.
Then $\mult(X_1,o)=-|H/H_1|Z\cdot Z$.
\end{prop}
\begin{proof}
Since $\cD$ satisfies the GCD condition,  $Z^1=\psi^*Z$ by
 \lemref{l:*gcd}.
Let $U$ and $D_i$ be as in the claim.
Then $\cO_{\bar X_1}(-\varphi^*D_i) \cong
 \varphi ^*\cO_{\tX}(-\sigma(D_i))$ and it is generated by global
 sections at every point of $\varphi^{-1}(U\cap E)$.
Since $H^0(\cO_{\bar X_1}(-\varphi^*D_i)) \subset
 H^0(\cO_{\bar X_1}(-Z^1))$ and $\cO_{U}(-\varphi^*D_i)
 = \cO_{U}(-Z^1)$, it follows that  
 $\cO_{\bar X_1}(-Z^1)$  is generated by global
 sections at  every point of $E$.
Hence the  assertion follows from \thmref{t:mult} and the
 formula $Z^1\cdot Z^1=(\deg \varphi)Z\cdot Z$. 
\end{proof}

\section{The splice quotients}\label{s:SQ}

In this section we give a brief introduction of splice
quotients (see \cite{nw-CIuac}, \cite{nw-HSL},
\cite{o.pg-splice} for more details) and discuss generators
of ideals of type $H^0(\cO_{\tX}(-\sigma(A)))$.

Let $\Gamma$ denote the weighted
dual graph of $E$.
Since $E$ is a tree of rational curves, $\Gamma$ and
 the intersection matrix $I(E)$ have the same information.
Let $\delta_v=(E-E_v)\cdot E_v$.
An irreducible component  $E_v$ (or a suffix $v$) is called
an {\itshape end} (resp. a {\itshape node}) if $\delta_v=1$
(resp. $\delta_v\ge 3$). 
Let $\cE$ (resp. $\cN$) denote the set of indices of ends
(resp. nodes).
A connected component of $E-E_v$ is called a {\itshape branch} of $E_v$.

\begin{defn}
Let $\cM=\sum _{w \in \cE}\Z_{\ge 0}E^*_w$,
 where $\Z_{\ge 0}$ is the set of nonnegative integers. We
 call an element of $\cM$  a {\itshape monomial cycle}.
For a monomial cycle $D=\sum _{i \in \cE}a_iE_i^*$, we
 associate a monomial  $z(D):=\prod_{i\in \cE} z_i^{a_i}$ of 
the power series ring $\C\{z\}:=\C\{z_i ; i \in \cE\}$.
\end{defn}

Let $p\: \tX^u \to \tX$ be the finite morphism in \proref{p:uac-rat}.

\begin{defn}
 For any $v \in \cV$, let $F_v=p^{-1}(E_v)$ and 
 define  the  {\itshape $v$-degree} of a monomial $z(D)$ to be the
 coefficient of $F_v$ in $p^*D$.
Let $f=f_0+f_1 \in \C\{z\}$, where $f_0$ is a nonzero
 quasihomogeneous polynomial with respect to  
 the $v$-degree and  $f_1$ is a series in monomials of
 $v$-degrees greater than the $v$-degree of $f_0$. 
Then we call $f_0$ the {\itshape $v$-leading form} of $f$, and 
 denote it by $\LF_{v}(f)$.
The  $v$-degree of $f_0$ is called the {\itshape $v$-order} of $f$.
\end{defn}

\begin{defn}\label{c:A}
We say that $\Gamma$ satisfies the {\itshape monomial condition} if for  any
 branch $C$ of any node $E_v$, 
 there exists a monomial cycle $D$
 such that $D-E^*_v$ is an effective integral cycle
 supported on $C$.
The monomial $z(D)$ is called an {\itshape admissible
 monomial} belonging to  $C$. 
\end{defn}

\begin{defn}\label{d:N-Wsystem}
 Assume that the monomial condition is satisfied.
Let $E_v$ be a node with branches $C_1, \dots , C_{\delta_v}$,
and let  $m_i$ denote an admissible monomial belonging to
 $C_i$.
Let $(c_{ij})$  
be an arbitrary $(\delta_v-2) \times  \delta_v$ matrix with 
 $c_{ij} \in \C$ such  
 that every maximal minor of it has rank $\delta_v-2$.
We define polynomials $f_1, \ldots ,f_{\delta_v-2}$ by $f_i=\sum
 _{j=1}^{\delta_v}c_{ij}m_j$.  
Let $\cF_v=\{f_1 , \dots ,f_{\delta_v-2}\}$.
We call the set $\bigcup _{v \in \cN}\cF_v$ a {\itshape Neumann--Wahl system} associated
 with $\Gamma$.
\end{defn}

The group  $H$ acts on  $\C\{z\}$ as follows.
 For any monomial cycle $D$ and any element $h \in H$, we
 define
$$
h \cdot z(D)=(\Theta(D)(h))z(D).
$$
 Note that the set $\cF_v$ consists of $\Theta(E_v^*)$-eigenfunctions.

\begin{defn}[see {\cite[\S 7]{nw-CIuac}}]\label{d:sq}
Suppose that a set 
$$
 \cF:=\defset{f_{vj_v}}{v \in \cN, \; j_v=1, \dots  ,\delta_v-2} \subset \C\{z\}
$$ 
satisfies the following:
\begin{itemize}
 \item the set
 $\bigcup _{v \in \cN}\defset{\LF_v(f_{vj_v})}{j_v=1, \dots  ,\delta_v-2}$
is a Neumann--Wahl system  associated with $\Gamma$ 
(in this case, $\cF$ is called a system of splice diagram functions);
 \item every $f_{vj_v}$ is a $\Theta(E_v^*)$-eigenfunction. 
\end{itemize}
Then $Y:=\defset{z=(z_i)\in \C^{\#\cE}}{f(z)=0,
 \;  f \in \cF}$ is an isolated  complete intersection singularity having the action of $H$, and 
the quotient $Y/H$ is a normal surface singularity.
The singularity $(Y/H,o)$ is called  a {\itshape  splice-quotient singularity}. 
\end{defn}

Next we describe a characterization of splice quotients.

\begin{defn}\label{d:ecc}
 We say that $\tX$ satisfies the {\itshape end curve condition} if  for
 each $i\in \cE$, there exists $x_i\in 
 H^0(\cO_{\tX}(-\sigma(E_i^*)))$  such that $(x_i)_E=E_i^*$.
We call  $C_i:=\di (x_i)-\sigma(E_i^*)$ an {\itshape end
 curve} at $E_i$.
Note that  $C_i\cdot E=C_i\cdot E_i=1$.
\end{defn}

\begin{rem}
If $\tX$ satisfies the end curve condition,  so do the
 minimal good resolution and  
 any resolution obtained by the
 blowing-up of $\tX$ at singular points of $E$ or  points of
the intersection of  $E$ and end curves.
It is easy to see that  $X$ has a good resolution satisfying  the end curve condition if and only if its minimal good resolution
 satisfies the condition.
\end{rem}

If $\tX$ satisfies the end curve condition, then we obtain an 
 $H$-equivariant homomorphism of $\C$-algebras  
$$
 \psi\: \C\{z\} \to \cO_{{X^u},o}, \quad \psi
 (z_i)=x_i.
$$
The following theorem is due to Neumann--Wahl \cite{nw-ECT}
 (cf. \cite{o.ECT}, \cite{braun-2008}).
\begin{thm}[End Curve Theorem]
The singularity $\X$ is a splice quotient if and only if
 there exists a good resolution  satisfying the end curve condition.
In this case, the homomorphism $\psi$ is surjective and its kernel
 is generated by a system of splice diagram functions. 
\end{thm}

We assume that $\tX$ satisfies the end curve condition and use the
notation of \defref{d:sq}; thus
$(X^u,o)=\{f=0, f\in \cF\}\subset (\C^{\#\cE},o)$.

For each $w\in \cV$, we define  the {\itshape $w$-filtration}
$\{I_n^w\}_{n\ge 0}$
of $\cO_{X^u,o}$ by 
$$
 I_n^w=\defset{\psi (f)}{\vord{w}(f)\ge n}.
$$
Let $G(w)$ denote the associated graded ring $\bigoplus_{n\ge
0}I^w_n/I^w_{n+1}$. By the definition of $w$-degree,
 $I^w_n\subset \left(\rho_*\cO_{\t
 X^u}(-nF_w)\right)_o$. We will show that these filtrations
 are the same. 

\begin{lem}[cf. {\cite[2.6]{nw-CIuac}}]\label{l:G}
 For every $w\in \cV$, the ring $G(w)$ is  a reduced complete
 intersection ring isomorphic to  $\C[z]/ I^w$, where $I^w$
 is the ideal generated by 
$\LF_w\cF:=\defset{\LF_w(f)}{f \in \cF}$.
\end{lem}
\begin{proof}
 In case $\delta_w\ge 3$, it follows from  \cite[2.6]{nw-CIuac}.
Assume that $\delta_v\le 2$.
To prove the isomorphism $G(w)\cong \C[z]/ I^w$, we only need to show that 
 $\LF_w\cF$ forms a regular
 sequence (see \cite[3.3]{nw-CIuac}). 
We may assume that for every 
 $v\in\cN$  the $v$-leading forms of $f_{vj_v}$   are expressed as  
$$
\LF_v (f_{vj_v}) =
m_{vj_v}+a_{vj_v}m_{v\delta_v-1}+b_{vj_v}m_{v\delta_v},
 \quad 
 1\le j_v \le \delta_v-2, 
$$
and  that the admissible monomial $m_{v\delta_v}$ belongs to the branch
 containing $E_w$.
Note that  higher terms with respect to $v$-degree are also higher
 terms  with respect to $w$-degree (cf. \cite[3.8]{o.uac-certain}).
Thus we have 
$$
 \LF_w (f_{vj_v}) =
m_{vj_v}+a_{vj_v}m_{v\delta_v-1}, \quad  v \in \cN, \;
 1\le j_v \le \delta_v-2.
$$
Let $E_1$  and $E_2$ be ends.
Suppose that  these ends are separated by $E_w$ if
 $\delta_w=2$, or that $w=1$ if $\delta_w=1$.
Then  the ideal of $\C[z]$ 
generated by $\LF_w\cF \cup
 \{z_1,z_2\}$ defines a zero-dimensional variety whose
 support is the origin, and hence  $\LF_w\cF$ is
 a regular sequence (cf. \cite[4.4]{o.uac-certain}).
We see also that for any $c_1, c_2 \in \C^*$,
 the zero-dimensional  variety defined by $\LF_w\cF \cup
 \{z_1-c_1,z_2-c_2\}$ is nonsingular. 
Therefore $\C[z]/ I^w$ is
 reduced
(cf. \cite[p. 711]{nw-CIuac}).
\end{proof}

\begin{lem}\label{l:multMC}
Let $w \in \cV$ and let $C_1, \dots, C_{\delta_w}$ be the branches
 of $E_w$. Then there exists a positive integer $n$ and
 monomial cycles $D_1, \dots, D_{\delta_w}$ such that
 $D_i-nE_w^*$   is an effective integral cycle
 supported on $C_i$ for every $i=1, \dots, \delta_w$.
\end{lem}
\begin{proof}
For every $1\le i  \le \delta_w$, let  $E_{w_i}$ be an end
 contained in  $C_i$.  Let $E_{w_i}^{\times}$ be the rational cycle
 supported on $C_i$ such that $E_{w_i}^{\times}\cdot
 E_j=-\delta_{w_i j}$ for all $E_j\subset C_i$. 
Take positive
integers  $n, n_1,  \dots, n_{\delta_w}$ such that 
$n_iE_{w_i}^{\times}\in L$ and $n=n_iE_{w_i}^{\times}\cdot E_w$ for
 every $i$. Then put  $D_i=n_iE_{w_i}^{\times}+nE_w^*$.
\end{proof}

In  \cite[4.5]{nem.coh-sq}, the following results are shown
as a corollary of the equivariant
Campillo--Delgado--Gusein-Zade formula.

\begin{prop}\label{p:filt}
 For every $w\in \cV$, we have $I^w_n=\left(\rho_*\cO_{\tX^u}(-nF_w)\right)_o$.
\end{prop}
\begin{proof}
The proof is the same as that of \cite[3.3]{o.pg-splice}. 
There are only two essential points for the proof:
$G(w)$ is reduced by \lemref{l:G};
 $\cO_{\tX}(- \sigma (mE_w^*))$
 is generated by global sections  for some $m\in \N$, which follows from 
 the end curve condition and \lemref{l:multMC}.
\end{proof}

\begin{cor}\label{c:filt-monom}
 For any $A\in L^*$, the space
 $H^0(\cO_{\tX}(-\sigma(A)))$ consists of the series of
 $\psi(z(D))$ with $D\ge A$ and $D-A\in L$. 
\end{cor}
\begin{proof}
By \cite[3.5]{o.pg-splice}, 
 $H^0(\cO_{\tX}(-\sigma(A)))$ is the $\Theta(A)$-eigenspace
 of  $H^0(\cO_{\tX^u}(-p^*A))$.
It follows from  \proref{p:filt} and the definition of the
 $w$-order that the space $H^0(\cO_{\tX^u}(-p^*A))$  consists of the series of
 $\psi(z(D))$ with $D\ge A$.
\end{proof}

\section{The multiplicity of splice quotients}\label{s:mult}

We assume that $\tX$   satisfies the end curve
 condition. Hence $X$ is a splice quotient.
 For each $i \in \cE$ we fix the section $x_i$ and  the  end
 curve $C_i$ in
 \defref{d:ecc}, and let  $C_i$ intersect $E$ at a point $b_i\in E_i$.
It is clear that  $\cO_{\tX}(-\sigma (E_i^*))$ is generated
 by global sections at every point of $E\setminus\{b_i\}$
 for every $i\in \cE$
and that the same is true for $\cO_{\tX}(-D)$ if $D$ is a monomial cycle.

For any $\Q$-cycle $D=\sum_{v\in \cV}c_vE_v$, we denote by $M_v(D)$ the coefficient of $E_v$ in $D$, i.e., $M_v(D)=c_v$.

\begin{prop}\label{p:freeend}
Let  $i \in \cE$.
Then   $\cO_{\tX}(-\sigma (E_i^*))$ is  generated
 by global sections at $b_i$ if and only if 
there exists a monomial cycle $D_i \in \sum_{j\in
 \cE\setminus\{i\}}\Z_{\ge 0}E_j^*$ such that $M_i(D_i)=M_i(E_i^*)$.
\end{prop}
\begin{proof}
We first note that  $\cO_{\tX}(-\sigma (E_i^*))$ is generated
 by global sections at $b_i$ if and only if the restriction map
$$
r\: H^0(\cO_{\tX}(-\sigma(E_i^*))) \to 
H^0(\cO_{E_i}(-\sigma(E_i^*)))
$$
is surjective. 
It is clear that  the
 dimension of 
 $H^0(\cO_{E_i}(-\sigma(E_i^*)))$ is two and $r(x_i)$ has a zero at $b_i$.

Suppose that there exists a monomial cycle $D_i$ with the
 property in the assertion.
If $E_{i'}$ denotes the irreducible component intersecting
 $E_i$, then $M_{i'}(D_i)=M_{i'}(E_i^*)+1$.
Hence $r(\psi(z(D_i)))$
 has a zero at $\{b_i'\}:=E_i\cap E_{i'}$, and   $r$ is
 surjective.

Assume that $r$ is surjective.
Then there exists a section $f\in H^0(\cO_{\tX}(-\sigma(E_i^*)))$ such
 that $r(f)$ has a zero at $b_i'$. 
The divisor $(f)_E\in L^*$ satisfies that
$(f)_E\ge E_i^*$,  $(f)_E-E_i^*\in
 L$, and 
 $M_i((f)_E)=M_i(E_i^*)$.
By \corref{c:filt-monom},
since $f\in H^0(\cO_{\tX}(-\sigma((f)_E)))$,
we may assume that $f$ is a linear combination of
 monomials  $\psi(z(D))$ with $D\ge (f)_E$ and $M_i(D)=M_i(E_i^*)$. 
Hence there exists a monomial cycle such as $D_i$.
\end{proof}

Let $\cB$ be the set of points $b_i$ being the base point of
$\cO_{\tX}(-\sigma(E_i^*))$. 

\begin{rem}\label{r:cB}
By \proref{p:freeend}, whether $b_i$ is in
$\cB$ or not depends only on $\Gamma$.
\end{rem}

 Let $\tau \:W\to \tX$ be the blowing-up with center $\cB$.
For each $i\in \cE$,  let $F_i=\tau^{-1}(b_i)$ if $b_i\in
\cB$, and  let $F_i=\tau^{-1}(E_i)$ otherwise.
Note that $\{F_i\}_{i\in \cE}$ is the set of all ends 
of $F:=\tau^{-1}(E)$.
On $W$,  we can also consider $F_i^*$,
$\sigma$, monomial cycles, and so on.

\begin{lem}\label{l:gen-end}
An invertible sheaf $\cO_{W}(-\sigma(F_i^*))$ is generated by global sections
 for every $i\in \cE$.  As a consequence, for every monomial
 cycle $D$ on $W$, $\cO_W(-\sigma (D))$ is also generated by
 global sections.
\end{lem}
\begin{proof}
If  $F_i=\tau^{-1}(E_i)$, then
 $\cO_{W}(-\sigma(F_i^*))\cong
 \tau^*\cO_{\tX}(-\sigma(E_i^*))$ is generated by global sections.
Assume that $b_i\in \cB$. Let $I_i\subset \cO_{\tX}$ be the
 ideal sheaf of the point $b_i$.
Since  the degree
 of $\cO_{E_i}(-\sigma(E^*_i))$ is one,  
$I_i\cO_{\tX}(-\sigma(E_i^*))$  is generated by global
 sections, thus so is $I_i\tau^*\cO_{\tX}(-\sigma(E_i^*))=
 \cO_{W}(-\tau^*\sigma(E_i^*)-F_i)$. 
Then the assertion follows from that $F_i^*=\tau^*E_i+F_i$.
\end{proof}

We shall compute the multiplicity of the
singularity $(X_1, o)$, where $H_1$ is a subgroup of $H$ and
$X_1=X^u/H_1$.
Let $\tX'\to X$ be the minimal good resolution.
We define the set $\cB'$ of points on $\tX'$  in the same
way as the set $\cB$ above.
Suppose that  $\tX\to \tX'$ is the blowing-up with center
       $\cB'$. Then $\cB=\emptyset$ by \lemref{l:gen-end}.
By the definition of the $H$-action, the invariant ring
$\C\{z\}^{H_1}$ consists of the 
series of monomials $z(D)$ with $D\in
\cM^{H_1}:=\Theta^{-1}(H_1^{\bot})\cap \cM$.
Let $\cM^{H_1}_{gen}$ denote  the set of  generators 
 of the semigroup $\cM^{H_1}$.
Applying \lemref{l:to-gcd},
we can take the birational morphism $\varphi\:W\to\tX$ such that
  $\varphi^*\cM^{H_1}_{gen}$ satisfies the GCD
  condition. 
It is clear that   for every $D\in \varphi^*\cM^{H_1}_{gen}$,
 $\cO_{W}(-\sigma(D))$ is generated by global sections.
Thus the set  $\varphi^*\cM^{H_1}_{gen}$
       will play the
role of the set $\cD=\{D_1, \dots, D_m\}$ in \proref{p:mult-gen}.

Let $Z=\gcd\varphi^* \cM^{H_1}_{gen}$. 

\begin{thm}\label{t:main}
In the situation above,  $\mult(X_1,o)=-|H/H_1|Z\cdot Z$. Consequently,
 $\mult(X_1,o)$ is determined by the weighted dual graph
 of $\X$ and $H_1$.
\end{thm}
\begin{proof}
The formula follows from  \proref{p:mult-gen}.
On the other hand, \remref{r:cB} shows that the
 resolution graph of $\tX$ can be  determined by that of $\tX'$,
 and whether the GCD condition is satisfied depends on the
 graph. Hence the  resolution graph of $W$ is determined by  the weighted dual graph
 of $\X$ and $H_1$, and so is $Z\cdot Z$.
\end{proof}

\begin{cor}
 The multiplicity of a splice quotient singularity and its
 universal abelian cover can be computed from the weighted
 dual graph.
\end{cor}

\section{Examples}\label{s:ex}

We will show some examples as an application of our result.
Suppose that $\X$ is a splice quotient and  $\tX$ a good resolution satisfying the end curve condition. 
Let $\Gamma$ be the weighted dual graph associated with $\tX$.
For a subgroup $H_1\subset H$, let $X_1=X^u/H_1$ and  $Z=\gcd
\cM^{H_1}$. Note that $\gcd\cM^{H_1}=\gcd\cM^{H_1}_{gen}$ and that if $\cM^{H_1}$ satisfies the GCD condition then so does  $\cM^{H_1}_{gen}$.
Therefore if we obtain a good resolution on which $\cM^{H_1}$ satisfies the GCD condition and $\cO(-Z)$ is generated by global sections, then we can compute the multiplicity without gaining the generators of $\cM^{H_1}$. 
The following remark will be helpful for our computation.

\begin{rem}\label{r:free}

(1) If $Z\cdot E_v=0$, then the set $\cM^{H_1}$ satisfies the GCD condition at every point of $E_v$.
Therefore, if the GCD condition is not satisfied at
 $E_v\cap E_w$, then $Z\cdot E_v$ and $Z\cdot E_w$ are negative.
This is shown as follows. Suppose that  $E_w\cap E_v\ne\emptyset$.
Let $D_v$ be an element of the set
$$
 \cM^{H_1}_v:=\defset{D\in \cM^{H_1}}{M_v(D)=M_v(Z)}
$$ such
 that $M_w(D_v)$ is the minimum of $M_w(\cM^{H_1}_v)$.
If $M_w(D_v)>M_w(Z)$, i.e., the GCD condition is not satisfied
 at $E_v\cap E_w$, then $Z\cdot E_v<D_v\cdot E_v\le 0$.			     

(2) Suppose $b_i\in \cB$. If there exists $D\in \cM^{H_1}_{i}$ such
 that $D\cdot E_i=0$, then the blowing-up at
 $b_i$ in the discussion of \sref{s:mult} can be
 omitted, because in this case $\cO_{\tX}(-\sigma(Z))$ is
 generated by global sections at the point $b_i$.
Note that if $Z\cdot E_i=0$ then such a cycle $D$ exists.
\end{rem}

We denote by $\gen{D_1, \dots, D_k}$ the subgroup of $H$ generated by
 $\alpha(D_1), \dots, \alpha(D_k)$, 
where  $\alpha\:L^*\to H=L^*/L$ is the projection.

\begin{ex}
Suppose that  $\tX$ is the minimal good resolution and $\Gamma$ is as
 follows:
 \begin{center}
 \setlength{\unitlength}{0.5cm}
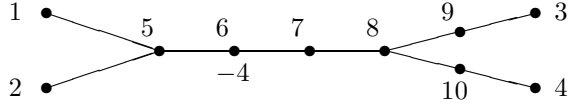
    \begin{picture}(13,3.5)(-2,-0.5)
 \multiput(-2,0)(0,2){2}{\ten}
 \put(-3,-0.2){$2$}
 \put(-3,1.8){$1$}
 \multiput(5,1)(2,0){2}{\ten}
 \put(0.5,1.4){$5$}
 \put(2.5,1.4){$6$}
 \put(2.5,0.2){$-4$}
 \put(4.5,1.4){$7$}
 \put(6.5,1.4){$8$}
 \put(8.5,1.8){$9$}
 \put(11.5,1.8){$3$}
 \put(11.5,-0.2){$4$}
 \put(8.5,-0.3){$10$}
 \put(3,1){\ten}
 \put(1,1){\ten}
 \multiput(9,0.5)(0,1){2}{\ten}
 \multiput(11,0)(0,2){2}{\ten}
 \put(1,1){\line(1,0){6}}
 \put(-2,0){\line(3,1){3}}
 \put(-2,2){\line(3,-1){3}}
 \put(7,1){\line(4,1){4}}
 \put(7,1){\line(4,-1){4}}
 \end{picture}
  \end{center}
 where the negative number indicates the weight of the
 vertex but the weights $-2$ are omitted, 
and the positive integers $i$  indicate $E_i$.
We may take the following equation for $X^u$:
 $$
 z_3z_4+z_1^2+z_2^2=z_3^3+z_4^3+z_1^5z_2^5=0.
 $$
 Since the leading forms (with $\deg z_i=1$) of these polynomials have
 degree 2 and 3 respectively  and forms a  regular
 sequence, we have
 $\mult (X^u,o)=6$.
 After a coordinate change,  the equation of $X=X^u/H$
 is represented as $z^2-x^6y^6-4xy(x+y)^3=0$ (cf. \cite[3.11]{o.pj}).
 Hence $\mult\X=2$.
We also see that $\X$ is an elliptic singularity (in the sense of \cite{wag.ell}) with $p_g\X=2$ (cf. \cite{nem.ellip}).

Now we show how we can use our method to compute the multiplicity of abelian covers of $\X$.
We denote an element $\sum_{i=1}^{10}a_iE_i$ by the sequence
$(a_1\; \cdots \; a_{10})$. 
Then 
$$
\left(\begin{array}{c}
 E_1^* \\ E_2^* \\ E_3^* \\ E_4^*  \\ E_5^*
 \end{array}\right)
=
\left(
\begin{array}{cccccccccc}
 1 & \frac{1}{2} & \frac{1}{2} & \frac{1}{2} & 1 & \frac{1}{2} & 1 &
   \frac{3}{2} & 1 & 1 \\
 \frac{1}{2} & 1 & \frac{1}{2} & \frac{1}{2} & 1 & \frac{1}{2} & 1 &
   \frac{3}{2} & 1 & 1 \\
 \frac{1}{2} & \frac{1}{2} & \frac{7}{3} & \frac{5}{3} & 1 & 1 & 3 &
   5 & \frac{11}{3} & \frac{10}{3} \\
 \frac{1}{2} & \frac{1}{2} & \frac{5}{3} & \frac{7}{3} & 1 & 1 & 3 &
   5 & \frac{10}{3} & \frac{11}{3} \\
 1 & 1 & 1 & 1 & 2 & 1 & 2 & 3 & 2 & 2
\end{array}
\right).
$$
 We see that  $H=\gen{E_1^*, E_3^*}$ and  $|H|=12$.
 (The Hermite normal form of $I(E)$ shows generators of $H$
 and their relations.)
If $H_1=\{0\}$, then $X_1=X^u$ and  
$$
Z=\gcd\{E_1^*, \dots, E_4^*\}=\frac{1}{2}E_5^*.
$$ 
By \remref{r:free}, $\cM^{H_1}$ satisfies the GCD condition 
and no blowing-up is needed. 
Thus  we can apply  \thmref{t:main} immediately; 
 $$
 \mult (X_1,o)=|H|(-Z\cdot Z)=12\cdot (1/2)=6.
 $$ 

If $H_1=\gen{E_1^*}$, then  $|H_1|=2$.
Since $D\cdot E_i^*=-M_i(D)$ for $D\in \cM$, it follows from the definition of $\cM^{H_1}$ that $\cM^{H_1}=\defset{D\in \cM}{M_1(D)\in \Z}$.
Thus we have 
$$
Z=\gcd\{E_1^*, E_2^*+E_3^*, E_3^*+E_4^*, E_4^*+E_2^*\}=E_1^*.
$$
 Since $E_2^*+E_3^*\in \cM^{H_1}$ and $M_1(E_2^*+E_3^*)=M_1(Z)$, 
again by \remref{r:free} we can apply  \thmref{t:main} immediately; 
 $$
 \mult (X_1,o)=|H/H_1|(-Z\cdot Z)=(12/2)\cdot 1=6.
 $$
 For other cases, we can easily compute the
 multiplicity of $(X_1,o)$ without blowing-ups; see Table \ref{fig:table}.

\begin{table}[ht]
$$
 \begin{array}{c|c|c|c|c}
 \hline
 H_1  & H_1^{\flat} & |H_1| & Z & \mult (X_1) \\
 \hline
   \{0\} & H &   1 &   (1/2)E_5^* &  6 \\

   \gen{E_1^*} & \gen{E_1^*, 2E_3^*} &   2 &  E_1^* &  6 \\
 
  \gen{3 E_3^*} & \gen{E_3^*} &  2 &  E_6^* &  6 \\
 
  \gen{E_1^*+3E_3^*} & \gen{E_1^*+E_3^*} &  2 &  E_2^* &  6 \\
 
  \gen{2E_3^*} & \gen{E_1^*, 3E_3^*} &  3 &  (1/2)E_5^* &  2 \\
 
  \gen{E_1^*, 3E_3^*} & \gen{2 E_3^*} &  4 &  E_5^* &  6 \\
 
  \gen{E_3^*} & \gen{3 E_3^*} &  6 &  E_5^* &  4 \\
 
   \gen{E_1^*, 2E_3^*} & \gen{E_1^*} &  6 &  E_1^* &  2 \\
 
  \gen{E_1^*+E_3^*} & \gen{E_1^*+3 E_3^*} &  6 &  E_2^* &  2 \\
 
  H & \{0\} &  12 &  E_5^* & 2 \\
 \hline
 \end{array}
 $$
\caption{}\label{fig:table}
\end{table}
\end{ex}

\begin{ex}
Next we suppose that  $\tX$ is the minimal good resolution and $\Gamma$ is as
 follows:
 \begin{center}
 \setlength{\unitlength}{0.5cm}
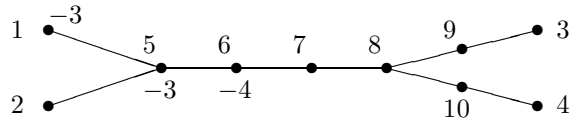
    \begin{picture}(13,3.5)(-2,-0.5)
 \multiput(-2,0)(0,2){2}{\ten}
 \put(-3,-0.2){$2$}
 \put(-3,1.8){$1$}
 \put(-2,2.2){$-3$}
 \multiput(5,1)(2,0){2}{\ten}
 \put(0.5,1.4){$5$}
 \put(0.5,0.2){$-3$}
 \put(2.5,1.4){$6$}
 \put(2.5,0.2){$-4$}
 \put(4.5,1.4){$7$}
 \put(6.5,1.4){$8$}
 \put(8.5,1.8){$9$}
 \put(11.5,1.8){$3$}
 \put(11.5,-0.2){$4$}
 \put(8.5,-0.3){$10$}
 \put(3,1){\ten}
 \put(1,1){\ten}
 \multiput(9,0.5)(0,1){2}{\ten}
 \multiput(11,0)(0,2){2}{\ten}
 \put(1,1){\line(1,0){6}}
 \put(-2,0){\line(3,1){3}}
 \put(-2,2){\line(3,-1){3}}
 \put(7,1){\line(4,1){4}}
 \put(7,1){\line(4,-1){4}}
 \end{picture}
  \end{center}
We can take the following  equation for $X^u$:
 $$
 z_3z_4+z_1^3+z_2^2=z_3^3+z_4^3+z_1^{14}z_2^{17}=0.
 $$
As in the previous example,  $\mult (X^u,o)=6$.
We see that $\X$ is not Gorenstein because $c_1(K_{\tX})\not\in L$, and $p_g\X=1$ by the formula for the geometric genus (see \cite{o.pg-splice}).

Let us  compute $\mult (X^u,o)$ using our method. So let $H_1=\{0\}$.
We have
$$
\left(\begin{array}{c}
 E_1^* \\ E_2^* \\ E_3^* \\ E_4^*  
 \end{array}\right)
=
 \left(
\begin{array}{cccccccccc}
 \frac{2}{5} & \frac{1}{10} & \frac{1}{10} & \frac{1}{10} &
   \frac{1}{5} & \frac{1}{10} & \frac{1}{5} & \frac{3}{10} &
   \frac{1}{5} & \frac{1}{5} \\
 \frac{1}{10} & \frac{13}{20} & \frac{3}{20} & \frac{3}{20} &
   \frac{3}{10} & \frac{3}{20} & \frac{3}{10} & \frac{9}{20} &
   \frac{3}{10} & \frac{3}{10} \\
 \frac{1}{10} & \frac{3}{20} & \frac{119}{60} & \frac{79}{60} &
   \frac{3}{10} & \frac{13}{20} & \frac{23}{10} & \frac{79}{20} &
   \frac{89}{30} & \frac{79}{30} \\
 \frac{1}{10} & \frac{3}{20} & \frac{79}{60} & \frac{119}{60} &
   \frac{3}{10} & \frac{13}{20} & \frac{23}{10} & \frac{79}{20} &
   \frac{79}{30} & \frac{89}{30}
\end{array}
\right).
$$
A direct  computation shows that
$$
|H|=60, \quad Z=\frac{1}{10}(E_1^*+3E_5^*), \quad
Z\cdot Z=-7/100.
$$
Then $|H|(-Z\cdot Z)=21/5\not\in \Z$. In fact, 
the set $\{E_1^*, \dots, E_4^*\}$
does not satisfy the GCD condition at $E_1\cap E_5$.
However, 
since  $M_1(E_1^*)=M_1(4 E_2^*)$, it follows that
 $b_1\not\in \cB$ by \proref{p:freeend}. 
Thus we need only blowing-ups over  $E_1\cap E_5$ (cf. \remref{r:free}) to obtain  a birational morphism $\varphi\:W\to\tX$
such that $\varphi^*\{\frac{2}{5}E_1+\frac{1}{5}E_5,
\frac{1}{10}E_1+\frac{3}{10}E_5\}$ satisfies the GCD condition (cf. \lemref{l:to-gcd}).
Then we obtain the following graph:
 \begin{center}
 \setlength{\unitlength}{0.5cm}
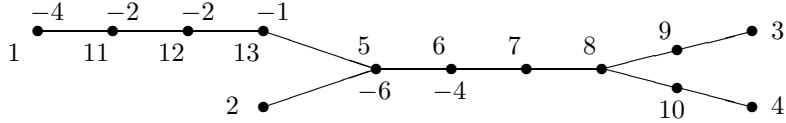
    \begin{picture}(19,3.5)(-8,-0.5)
 \multiput(-8,2)(2,0){3}{\ten}
 \put(-8,2){\line(1,0){6}}
 \multiput(-2,0)(0,2){2}{\ten}
 \put(-3,-0.2){$2$}
 \put(-8.8,1.2){$1$}
 \put(-6.8,1.2){$11$} 
\put(-4.8,1.2){$12$}
\put(-2.8,1.2){$13$}
 \put(-2.2,2.3){$-1$}
 \put(-4.2,2.3){$-2$}
 \put(-6.2,2.3){$-2$}
 \put(-8.2,2.3){$-4$}
 \multiput(5,1)(2,0){2}{\ten}
 \put(0.5,1.4){$5$}
 \put(0.5,0.2){$-6$}
 \put(2.5,1.4){$6$}
 \put(2.5,0.2){$-4$}
 \put(4.5,1.4){$7$}
 \put(6.5,1.4){$8$}
 \put(8.5,1.8){$9$}
 \put(11.5,1.8){$3$}
 \put(11.5,-0.2){$4$}
 \put(8.5,-0.3){$10$}
 \put(3,1){\ten}
 \put(1,1){\ten}
 \multiput(9,0.5)(0,1){2}{\ten}
 \multiput(11,0)(0,2){2}{\ten}
 \put(1,1){\line(1,0){6}}
 \put(-2,0){\line(3,1){3}}
 \put(-2,2){\line(3,-1){3}}
 \put(7,1){\line(4,1){4}}
 \put(7,1){\line(4,-1){4}}
 \end{picture}
  \end{center}
We may assume $\tX$ is a good resolution with the graph above.
Then we obtain that $Z=\frac{1}{10}E_{13}^*$ and $Z\cdot Z=-\frac{1}{10}$. 
By \thmref{t:main}  we have 
$\mult (X^u,o)=60\cdot\frac{1}{10}=6$.
\end{ex}


\providecommand{\bysame}{\leavevmode\hbox to3em{\hrulefill}\thinspace}
\providecommand{\MR}{\relax\ifhmode\unskip\space\fi MR }
\providecommand{\MRhref}[2]{%
  \href{http://www.ams.org/mathscinet-getitem?mr=#1}{#2}
}
\providecommand{\href}[2]{#2}

\end{document}